\documentclass[11pt]{amsart}
\usepackage{graphicx}
\usepackage[usenames]{color}
\usepackage{amsmath,amssymb,dsfont}
\usepackage{mathrsfs}
\newtheorem{thm}{Theorem}[section]

\newtheorem{lem}[thm]{Lemma}

\theoremstyle{definition}
\newtheorem{defn}[thm]{Definition}
\theoremstyle{remark}
\newtheorem{rem}[thm]{Remark}
\numberwithin{equation}{section}

\begin{document}

\title[Two-dimensional Invisibility Cloaking]{\bf Two dimensional invisibility cloaking via transformation optics}%

\author{Hongyu Liu}%
\address{Department of Mathematics, University of Washington, Box 354350, Seattle, WA, 98195--4350}%
\email{hyliu@math.washington.edu}%

\author{Ting Zhou}%
\address{Department of Mathematics, University of Washington, Box 354350, Seattle, WA, 98195--4350}%
\email{tzhou@math.washington.edu}%

\thanks{The work of HYL is partly supported by NSF grants, FRG DMS 0554571 and DMS 0758537}%
\keywords{invisibility cloaking, transformation optics, finite energy solutions, singularly weighted Sobolev space}%

\date{May 05, 2010}%
\begin{abstract}
We investigate two-dimensional invisibility cloaking via
transformation optics approach. The cloaking media possess much more
singular parameters than those having been considered for
three-dimensional cloaking in literature. Finite energy solutions
for these cloaking devices are studied in appropriate weighted
Sobolev spaces. We derive some crucial properties of the singularly
weighted Sobolev spaces. The invisibility cloaking is then justified
by decoupling the underlying singular PDEs into one problem in the
cloaked region and the other one in the cloaking layer. We derive
some completely novel characterizations of the finite energy
solutions corresponding to the singular cloaking problems.
Particularly, some `hidden' boundary conditions on the cloaking
interface are shown for the first time. We present our study for a
very general system of PDEs, where the Helmholtz equation underlying
acoustic cloaking is included as a special case.
\end{abstract}
\maketitle
\section{Introduction}

A region is said to be \emph{cloaked} if its contents together with
the cloak are indistinguishable from background space to certain
exterior detections. Blueprints for making objects invisible to
electromagnetic waves were proposed by Pendry {\it et al.}
\cite{PenSchSmi} and Leonhardt \cite{Leo} in 2006. In the case of
electrostatics, the same idea was discussed by Greenleaf {\it et
al.} \cite{GLU2} in 2003. The works \cite{GLU2,Leo,PenSchSmi} rely
on {\it transformation optics} for the construction of cloaking
devices, which we shall further examine in the present paper. For
state-of-the-art surveys on the rapidly growing literature and many
applications of transformation optics, we refer to
\cite{GKLU5,Nor,YYQ}.

In transformation optics, the key ingredient is that optical
material parameters have transformation properties which could be
{\it pushed forward} to form new material parameters. Then, to
construct cloaking devices, the idea is to \emph{blow up} a point in
the background space to form the cloaked region. The ambient
background medium is then \emph{pushed forward} to form the cloaking
medium. Since the blowing-up transformation is singular, the
resulting cloaking medium is inevitably singular. Two theoretical
approaches have come out to handle the singular cloaking problems.
In Kohn {\it et al.} \cite{KSVW}, the authors introduce the notion
of {\it near-invisibility} cloaking in electrostatics from a
regularization viewpoint. The singular ideal cloaking is regarded as
the limit of the regular near-invisibility cloaking depending on
certain regularizer. For acoustic cloaking, the near-invisibility is
investigated in \cite{KOVW,Liu,Ngu} in both two and three
dimensions. In \cite{GKLU}, Greenleaf {\it et al.} proposed to
investigate the physically meaningful solutions, i.e. {\it finite
energy solutions}, corresponding to degenerate differential
equations underlying the three-dimensional cloaking. The proposal
has been shown to work for both acoustic and electromagnetic
cloaking, and can treat cloaking of passive objects as well as
active/radiating objects. On the other hand, the analysis in
\cite{GKLU} is conducted in the geometric setting by taking
advantage of the one-to-one correspondence in $\mathbb{R}^3$ between
the optical parameters and a smooth Riemannian metric. This argument
does not carry over to $\mathbb{R}^2$. As one shall see in Section
2, for two-dimensional cloaking, the cloaking medium has both
degeneracy and blow-up singularities at the cloaking interface,
making the problem more difficult to analyze.

In this paper, we consider the two-dimensional invisibility cloaking
for a very general system of second order partial differential
equations. The Helmholtz equation underlying the acoustic cloaking
is included as a special case. In order to handle the singular
cloaking problem, we follow the finite energy solutions approach
from \cite{GKLU}. Recently, Hetmaniuk and Liu \cite{HetmaniukLiu}
introduce weighted Sobolev spaces with degenerate weights for
three-dimensional acoustic cloaking problems, encompassing and
generalizing the idea of finite energy solutions approach. For the
present two-dimensional cloaking problems, we study weighted Sobolev
spaces with more severely singular weights. The invisibility
justification follows by study of weak solutions from the introduced
weighted Sobolev space to the underlying singular PDEs. Our analysis
is given in a very general setting. The cloaking is shown to work
for arbitrary positive frequency and can cloak both passive media
and source/sink inside the cloaked region. Since the cloaking media
are more singular than those for the three-dimensional cloaking, we
derive completely different and novel characterizations of solutions
to the underlying wave equations compared to those derived for the
three-dimensional acoustic cloaking in \cite{GKLU,HetmaniukLiu}.
Moreover, some `hidden' boundary conditions on the cloaking
interface are shown for the first time, giving more insights into
the invisibility cloaking.

In this paper, we focus entirely on transformation-optics-approach
in constructing cloaking devices. But we mention in passing the
other promising cloaking schemes including the one based on
anomalous localized resonance \cite{MN}, and another one based on
special (object-dependent) coatings \cite{AE}. The rest of the paper
is organized as follows. In Section 2, we introduce the
transformation optics and invisibility cloaking and give the
construction of the two-dimensional radial cloaking devices.
Section~3 is devoted to the analysis of the radial cloaking by
considering finite energy solutions in singularly weighted Sobolev
spaces. In Section~4, we extend our study to general invisibility
cloaking.

\section{Transformation optics and radial invisibility
cloaking}\label{sect:2}

We first fix notations for some function spaces which are crucial
for our study. Let $\Omega$ be a bounded Lipshitz domain in
$\mathbb{R}^2$. Let $m\geq 1$ be an integer, and
$u:\Omega\rightarrow\mathbb{C}^m$ be a complex vector-valued
function. $L^p(\Omega)^m$ is the space consisting of
$\mathbb{C}^m$-valued measurable functions whose components belong
to $L^p(\Omega)$. Following Schwartz, put
$\mathscr{E}(\Omega)^m=C^\infty(\Omega)^m$ and
$\mathscr{D}(\Omega)^m=C_{comp}^\infty(\Omega)^m$ be complex-valued,
smooth test function spaces. Let $H^s(\Omega)$ denote the standard
Sobolev space of order $s$ and
\[
\widetilde{H}^s(\Omega)=\mbox{closure of $\mathscr{D}(\Omega)$ in
$H^s(\mathbb{R}^n)$}.
\]
Note that $\widetilde{H}^{-s}(\Omega)$ is an isometric realization
of $H^s(\Omega)^*$ for $s\in\mathbb{R}$. The definition of the
vector Sobolev spaces on $\Omega$, $H^s(\Omega)^m=H^s(\Omega;
\mathbb{C}^m)$ and
$\widetilde{H}^s(\Omega)^m=\widetilde{H}^s(\Omega;\mathbb{C}^m)$
etc. shall now be obvious.

Next, we introduce the second-order partial differential operator
(PDO) $\mathcal{P}$ of the form
\begin{equation}\label{eq:pdo}
\mathcal{P}u=-\sum_{\alpha,\beta=1}^2\frac{\partial}{\partial
x_\alpha}\left(A^{\alpha\beta}\frac{\partial u}{\partial
x_\beta}\right)-\omega^2 Bu\ \ \mbox{on $\Omega$},
\end{equation}
where the coefficients
\begin{equation}
A^{\alpha\beta}=[a_{pq}^{\alpha\beta}]\in L^\infty(\Omega)^{m\times
m},\ \ B=[b_{pq}]\in L^\infty(\Omega)^{m\times m}
\end{equation}
are functions from $\Omega$ into $\mathbb{C}^{m\times m}$, the space
of complex $m\times m$ matrices. Here $\omega\in\mathbb{R}$ and
$m\geq 1$ is an integer and thus, $\mathcal{P}$ acts on a (column)
vector-valued function $u:\Omega\rightarrow \mathbb{C}^m$ to give a
vector-valued function $\mathcal{P}u:\Omega\rightarrow\mathbb{C}^m$,
whose components are
\[
(\mathcal{P}u)_p=-\sum_{k=1}^m\sum_{\alpha,\beta=1}^2
\partial_\alpha\left(a_{pk}^{\alpha\beta}\partial_\beta
u_k\right)-\omega^2 \sum_{q=1}^m b_{pq}u_q.
\]
In the sequel, let $A:=[A^{\alpha\beta}]\in\mathbb{C}^{2m\times 2m}$
be the block matrix, and $\mathcal{P}_{[A,B]}$ be the PDO
(\ref{eq:pdo}) associated with $A$ and $B$. For the present study,
we always assume that
\begin{equation}\label{eq:symmetric}
(A^{\alpha\beta})^*=A^{\beta\alpha}\ \ \mbox{for\ \
$\alpha,\beta=1,2$},
\end{equation}
where $*$ denotes the conjugate transpose of a matrix or vector.
Moreover, we introduce the following algebraic conditions for the
coefficient matrices $A^{\alpha\beta}$ and $B$. Let
$0<c_1<c_2<\infty$ be two constants. For all $x\in\Omega$, and
arbitrary $\xi_1,\xi_2\in\mathbb{C}^m\ \mbox{and}\
\eta\in\mathbb{C}^m$, we have
\begin{align}
&c_1\sum_{l=1}^2|\xi_l|^2\leq\Re\sum_{\alpha,\beta=1}^2[A^{\alpha\beta}(x)\xi_\beta]^*\xi_\alpha\leq
c_2\sum_{l=1}^2|\xi_l|^2,\label{eq:algebraic}\\
&\qquad\qquad c_1|\eta|^2\leq \Re[B\eta]^*\eta\leq
c_2|\eta|^2.\label{eq:algebraic2}
\end{align}
Next, we associate $\mathcal{P}_{[A,B]}$ with a sesquilinear form
$\mathcal{Q}_{[A,B]}$, defined by
\begin{equation}\label{eq:sesqui}
\mathcal{Q}_{[A,B]}(u,v)=\int_\Omega\left(\sum_{\alpha,\beta=1}^2(A^{\alpha\beta}\partial_\beta
u)^*\partial_\alpha v-(Bu)^*v\right) dx.
\end{equation}

Now, we are ready to present the PDE system for our study,
\begin{equation}\label{eq:pde system}
\mathcal{P}_{[A,B]} u= f\ \ \mbox{on\ $\Omega$},
\end{equation}
where $f\in \widetilde{H}^{-1}(\Omega)^m$. The Dirichlet-to-Neumann
(DtN) map
\begin{equation}\label{eq:dtn}
\Lambda_{A,B,f}^\omega:\ H^{1/2}(\partial\Omega)^m\rightarrow
H^{-1/2}(\partial\Omega)^m,
\end{equation}
associated with (\ref{eq:pde system}) is defined by
\begin{equation}\label{eq:dtn2}
\Lambda_{A,B,f}^\omega(h)=\sum_{\alpha,\beta\leq 2}\nu_\alpha
A^{\alpha\beta}\gamma(\partial_\beta u)\ \ \mbox{on \
$\partial\Omega$},
\end{equation}
where $\nu=(\nu_\alpha)_{\alpha=1}^2$ is the outward unit normal of
$\partial\Omega$, $u\in H^1(\Omega)^m$ solves (\ref{eq:pde system})
with $u|_{\partial\Omega}=h$ and $\gamma$ is the trace operator for
$\Omega$. The weak solution in (\ref{eq:dtn2}) is variationally
given by
\begin{equation}\label{eq:var form}
\mathcal{Q}_{[A,B]}(u,\varphi)=\langle
f,\varphi\rangle_\Omega:=\int_{\Omega} f^*\varphi\ dx.
\end{equation}
Due to (\ref{eq:algebraic}) and (\ref{eq:algebraic2}), we know that
$\mathcal{Q}_{[A,B]}$ (and so is $\mathcal{P}_{[A,B]}$) is {\it
coercive} on $H^1(\Omega)^m$ in the sense that
\begin{equation}\label{eq:coercivity}
\Re\mathcal{Q}(u,u)\geq
c_1\|u\|_{H^1(\Omega)^m}^2-c_2\|u\|_{L^2(\Omega)^m}^2\quad\mbox{for\
$u\in H^1(\Omega)^m$}.
\end{equation}
Hence, (\ref{eq:var form}) is uniquely solvable except at a discrete
set of eigenvalues for $\omega^2$ (see, e.g. \cite{McL}). In fact,
we know that $\Lambda_{A,B,f}^\omega$ is a well-defined continuous
and invertible operator provided (\ref{eq:algebraic}) and
(\ref{eq:algebraic2}) are satisfied and $\omega^2$ avoids the
(discrete set of) eigenvalues (cf. \cite{McL}).

In the case $m=1$ and $A, B$ are both real, (\ref{eq:pde system}) is
the scalar Helmholtz equation.
%
%
It describes the time-harmonic solutions $p(x)=u(x)e^{-i\omega t}$
of the scalar wave equation $p_{tt}-B^{-1}\nabla\cdot(A\nabla
p)=f(x)e^{-i\omega t}$. Here $f$ represents a source/sink inside the
region $\Omega$. $A$ and $B$ are the acoustic material parameters of
the medium supported in $\Omega$, related respectively to, density
tensor and modulus. For a {\it regular} acoustic medium,
(\ref{eq:algebraic}) and (\ref{eq:algebraic2}) are the physical
condition on the material parameters. According to our earlier
discussion, these are also mathematical conditions to guarantee the
well-posedness of the underlying Helmholtz equation. In the sequel,
we let $\{\Omega; A, B, f\}$ denote the medium and the source/sink
supported in $\Omega$. This is the prototype problem of our present
study. Moreover, we are concerned with the inverse problems of
identifying the inside object $\{\Omega; A, B, f\}$ by the exterior
wave measurements, which are encoded into the DtN operator
(\ref{eq:dtn}). The inverse problems have widespread practical
applications in science and engineering, and have received extensive
and intensive investigations in last years
(see, e.g., \cite{Isa,U}). In this context, an invisibility cloaking
device is introduced as follows (see also \cite{GKLU,KOVW}).

\begin{defn}\label{def:cloaking device}
Let $\Omega$ and $D$ be bounded domains in $\mathbb{R}^2$ with
$D\Subset\Omega$. $\Omega\backslash\bar{D}$ and $D$ represent,
respectively, the cloaking region and the cloaked region.
$\{\Omega\backslash\bar{D}; A_c,B_c, f_c\}$ is said to be an {\it
invisibility cloaking device} for the region $D$ with respect to the
regular reference/background space $\{\Omega; A_b,B_b,f_b\}$ if
\[
\Lambda_{A_e,B_e,f_e}^\omega=\Lambda_{A_b,B_b,f_b}^\omega\ \
\mbox{for all $\omega>0$},
\]
where the extended medium $\{\Omega; A_e,B_e\}$ and the extended
source $f_e$ are given by
\[
\{\Omega;A_e,B_e, f_e\}=\begin{cases} \{\Omega\backslash\bar{D};
A_c,B_c,f_c\}\ &\mbox{in
\ $\Omega\backslash\bar{D}$}, \\
\{D; A_a,B_a,f_a\}\ &\mbox{in $\ D$},
\end{cases}
\]
with $\{D; A_a,B_a,f_a\}$ arbitrary but regular.
\end{defn}

In Definition~\ref{def:cloaking device}, $\{\Omega; A, B\}$ is
regular means that $A$ and $B$ satisfy the algebraic conditions
(\ref{eq:algebraic}) and (\ref{eq:algebraic2}). According to
Definition~\ref{def:cloaking device}, we note that the cloaking
device $\{\Omega\backslash\bar{D}; A_c,B_c,f_c\}$ makes the interior
target object $\{D; A_a,B_a,f_a\}$ indistinguishable from the
reference/background space $\{\Omega; A_b,B_b,f_b\}$ by exterior
detections. In fact, the exterior observer would not even be aware
that something is being hidden.

Next, we present the transformation properties of material
parameters, which are the cruxes of the construction of cloaking
devices via transformation optics approach. In the following, we let
$y=F(x)$ be a diffeomorphism such that $F:\
\Omega\rightarrow\widetilde\Omega$. Then, the push-forwards of
material parameters are given by
\[
F_*\{\Omega; A, B\}=\{\widetilde\Omega; F_*A,
F_*B\}:=\{\widetilde\Omega; \tilde A, \tilde B\},
\]
with
\begin{align}
&\tilde{A}^{pq}(y)={\sum_{\alpha,\beta=1}^2\frac{\partial y_p}{\partial x_\beta}\frac{\partial y_q}
{\partial x_\alpha} A^{\alpha\beta}(x)} J^{-1}\bigg|_{x=F^{-1}(y)},\ \ p,q=1,2,\label{eq:formula 1}\\
&\tilde{B}(y)=B(x)J^{-1}|_{x=F^{-1}(y)},\label{eq:formula 2}
\end{align}
where $J:=\mbox{\textnormal{det}}([\partial y_\alpha/\partial
x_\beta])$, the determinant of the Jacobian of $F$.
We shall make use of the following result
\begin{lem}\label{lem:trans opt}
Let  $F:\ \Omega\rightarrow\widetilde\Omega$ be a diffeomorphism.
For $u, v\in H^1(\Omega)^m$, let $\tilde{u}=(F^{-1})^*u:=u\circ F\in
H^1(\widetilde\Omega)^m$ and $\tilde{v}=(F^{-1})^*v\in
H^1(\widetilde{\Omega})^m$. Then we have
\begin{equation}
\mathcal{Q}_{[A,B]}(u,v)=\mathcal{Q}_{[F_*A, F_*B]}(\tilde u, \tilde
v).
\end{equation}
\end{lem}

\begin{proof}
It is verified directly by change of variables in integrations as
follows
\begin{align*}
\mathcal{Q}_{[A,B]}(u,v)=&\int_{\Omega}\left(\sum_{\alpha,\beta=1}^2[A^{\alpha\beta}\frac{\partial u}
{\partial x_\beta}]^*\frac{\partial v}{\partial x_\alpha}-(Bu)^*v\right) dx\\
=&\int_{\widetilde\Omega}\sum_{p,q=1}^2\bigg[\left(\sum_{\alpha,\beta=1}^2\frac{\partial
y_p}{\partial x_\beta}\frac{\partial y_q}{\partial
x_\alpha}A^{\alpha\beta}/J\right)\frac{\partial \tilde u}{\partial
y_p}\bigg]^*\frac{\partial\tilde v}{\partial y_q}-(\frac B J\tilde
u)^*\tilde v\ dy\\
=& \mathcal{Q}_{[\tilde A, \tilde B]}(\tilde u, \tilde v).
\end{align*}

The proof is completed.

\end{proof}

In the rest of this section, based on the above transformation
properties, we give the construction of the two-dimensional cloaking
devices, which we shall investigate in subsequent sections.
We start our study by considering the cloaking of the unit central
disc. In Section 4, we shall indicate how to extend our study to the
general case. In the following, we denote by $\mathbf{B}_\rho$ the
central disc of radius $\rho$ and $\mathbf{S}_\rho:=\partial
\mathbf{B}_\rho$. Let $\{\mathbf{B}_2; A_b, B_b, f_b\}$ be the
(regular) background/reference space, where $f_b$ is supported away
from $\{0\}$. Here, we always assume that there is no eigenvalue
problem for (\ref{eq:var form}) in the reference space, and hence
there is a well-defined DtN operator on $\mathbf{S}_2$, namely
$\Lambda_{A_b, B_b, f_b}^\omega$. Consider the transformation $F$,
defined by
\begin{equation} \label{eqn:F} F: \left\{
\begin{array}{rcl}
\mathbf{B}_2\backslash \{0\} & \rightarrow & \mathbf{B}_2\backslash
\bar{\mathbf{B}}_1
\\
x & \mapsto & (1+\frac{|x|}{2})\frac{x}{|x|}
\end{array}
\right .
\end{equation}
$F$ blows up the origin in the reference space to $\mathbf{B}_1$
while maps $\mathbf{B}_2\backslash\{0\}$ to
$\mathbf{B}_2\backslash\bar{\mathbf{B}}_1$ and keeps $\mathbf{S}_2$
fixed. We note that the blow-up transformation (\ref{eqn:F}) has
been extensively investigated for the design of three-dimensional
cloaking devices in the literature (see \cite{GKLU5}). Under the
transformation $F$, the ambient reference medium in
$\mathbf{B}_2\backslash\{0\}$ is then push-forwarded to form
transformation medium in $\mathbf{B}_2\backslash\bar{\mathbf{B}}_1$
as follows
\begin{equation}\label{eq:cloakmedium}
A_c(y)=F_*A_b(x)\ \ \mbox{and}\ \ B_c(y)=F_*B_b(x),\quad y\in
\mathbf{B}_2\backslash\bar{\mathbf{B}}_1.
\end{equation}
In Section 3, we shall show that
\begin{thm}\label{thm:1}
$\{\mathbf{B}_2\backslash\bar{\mathbf{B}}_1; A_c, B_c, f_c\}$ with
$f_c=J^{-1}(F^{-1})^*f_b$ and $A_c$ and $B_c$ given by
(\ref{eq:cloakmedium}) is an invisibility cloaking device for the
region $\mathbf{B}_1$ with respect to the reference space
$\{\mathbf{B}_2; A_b, B_b,$ $ f_b\}$. That is, for any extended
object
\[
\{\mathbf{B}_2;A_e,B_e, f_e\}=\begin{cases}
\{\mathbf{B}_2\backslash\bar{\mathbf{B}}_1; A_c,B_c,f_c\}\ &\mbox{in
\ $\mathbf{B}_2\backslash\bar{\mathbf{B}}_1$}, \\
\{\mathbf{B}_1; A_a,B_a,f_a\}\ &\mbox{in $\ \mathbf{B}_1$},
\end{cases}
\]
where $\{\mathbf{B}_1; A_a, B_a\}$ is an arbitrary regular medium
and $f_a\in {\widetilde H}^{-1}(\mathbf{B}_1)^m$, we have
\begin{equation}\label{eq:equalcloaking}
\Lambda_{A_e,B_e,f_e}^\omega=\Lambda_{A_b,B_b,f_b}^\omega.
\end{equation}
\end{thm}

In the sequel, we conveniently define the push-forward of the source
term as
\[
F_*f=\mbox{\textnormal{det}}(DF)(F^{-1})^* f.
\]
So, the cloaking device in Theorem~\ref{thm:1} is given by
\[
\{\mathbf{B}_2\backslash\bar{\mathbf{B}}_1; A_c, B_c,
f_c\}=F_*\{\mathbf{B}_2\backslash\{0\}; A_b, B_b, f_b\}.
\]
Next, we derive the explicit expressions of material parameters of
the cloaking medium in Theorem~\ref{thm:1}. By (\ref{eq:formula
1})-(\ref{eq:formula 2}) and straightforward calculations, we have
for $y\in \mathbf{B}_2\backslash\bar{\mathbf{B}}_1$
\begin{align}
A_c(y)=&\frac{|y|-1}{|y|}\Pi(y)\otimes I_m A_b(F^{-1}(y))\Pi(y)\otimes I_m\label{eq:cloak1}\\
 &+\Pi(y)\otimes A_b(F^{-1}(y))(I-\Pi(y))\otimes I_m\nonumber\\
 &+\frac{|y|}{|y|-1}(I-\Pi(y))A_b(F^{-1}(y))(I-\Pi(y))\nonumber\\
 &+(I-\Pi(y))\otimes I_m A_b(F^{-1}(y))\Pi(y)\otimes I_m,\nonumber\\
B_c(y)=&\frac{4(|y|-1)}{|y|}B_b(F^{-1}(y)),\label{eq:cloak2}
\end{align}
where $I_m$ is the $m\times m$ identity matrix and $\Pi(y):
\mathbb{R}^2\rightarrow\mathbb{R}^2$ is the projection to the radial
direction, defined by
\[
\Pi(y)\xi=\bigg(\xi\cdot\frac{y}{|y|}\bigg)\frac{y}{|y|},
\]
i.e. $\Pi(y)$ is represented by the symmetric matrix $|y|^{-2}yy^T$.
For $y\in\mathbf{B}_2\backslash\bar{\mathbf{B}}_1$, we let
$\hat{y}:=y/|y|=[\hat{y}_1,\hat{y}_2]^T\in\mathbf{S}_1$ and
$\hat{y}^{\perp}=[\hat{y}^{\perp}_1,
\hat{y}^{\perp}_2]^T\in\mathbf{S}_1$ be such that
$\hat{y}\cdot\hat{y}^{\perp}=0$. Let $\mathbf{1}_m\in\mathbb{R}^m$
with each entry being 1, and set
$\xi_l=\hat{y}_l\mathbf{1}_m\in\mathbb{R}^m$ and
$\xi_l^\perp=\hat{y}_l^\perp\mathbf{1}_m\in\mathbb{R}^m$ for
$l=1,2$. Then, by straightforward calculations, together with the
following facts
\[
\Pi(y)\hat{y}=\hat{y}\quad(I-\Pi(y))\hat{y}=0,\quad\Pi(y)\hat{y}^\perp=0,\quad
(I-\Pi(y))\hat{y}^\perp=\hat{y}^\perp,
\]
we have
\begin{align*}
&\sum_{p,q=1}^2[A_c^{pq}(y)\xi_q]^*\xi_p=\frac{|y|-1}{|y|}\sum_{p,q=1}^2[A_b^{pq}(F^{-1}(y))\xi_q]^*\xi_p,\\
&\sum_{p,q=1}^2[A_c^{pq}(y)\xi_q^\perp]^*\xi_p^\perp=\frac{|y|}{|y|-1}\sum_{p,q=1}^2[A_b^{pq}(F^{-1}(y))\xi_q^\perp]^*\xi_p^\perp.
\end{align*}
Since $A_b$ is a regular medium parameter, we have by
(\ref{eq:algebraic}) that as $|y|\rightarrow 1^+$
\begin{align*}
\sum_{p,q=1}^2[A_c^{pq}(y)\xi_q]^*\xi_p\rightarrow
0\quad\mbox{and}\quad
\sum_{p,q=1}^2[A_c^{pq}(y)\xi_q^\perp]^*\xi_p^\perp\rightarrow\infty.
\end{align*}
Meanwhile, it is obvious that $B_c$ does not satisfy
(\ref{eq:algebraic2}). That is, the algebraic conditions
(\ref{eq:algebraic}) and (\ref{eq:algebraic2}) for a regular medium
are violated by the cloaking medium
$\{\mathbf{B}_2\backslash\bar{\mathbf{B}}_1; A_c, B_c\}$, which
exhibits both degeneracy and blow-up singularities as one approaches
the cloaking interface $\mathbf{S}_1^+$, i.e. from
$\mathbf{B}_2\backslash\bar{\mathbf{B}}_1$. We remark that this is
in sharp difference from the study in \cite{GKLU} for
three-dimensional cloaking devices where one would encounter only
degeneracy singularities.

\section{Finite energy solutions for singular PDEs}

Consider the differential equation underlying the cloaking problem
in Theorem~\ref{thm:1},
\begin{equation}\label{eq:cloaking Helmholtz}
\mathcal{P}_{[A_e,B_e]} u=f_e\quad\mbox{on\ $\mathbf{B}_2$},\quad
u|_{\mathbf{S}_2}=h.
\end{equation} As we have seen in the end of Section 2, the
cloaking medium $\{\mathbf{B}_2\backslash\bar{\mathbf{B}}_1;
A_c,B_c\}$ is singular. One has to be careful in defining the
meaning of a solution to the singular PDEs system~(\ref{eq:cloaking
Helmholtz}) (see also Remark~\ref{rem:necessity} for the necessity
of introducing a suitable class of weak solutions to
(\ref{eq:cloaking Helmholtz}) other than spatial $H^1$-solutions).
To that end, we define for $\phi(y)\in \mathscr{E}(\mathbf{B}_2)^m$,
\[
\mathcal{E}_{[A,B]}(\phi)=\bigg|\int_{\Omega}\left(\sum_{\alpha,\beta=1}^2(A^{\alpha\beta}\partial_\beta\phi)^*\partial_\alpha\phi
+(B\phi)^*\phi\right)\ dy\bigg|^{1/2}.
\]
We note that if $\{\mathbf{B}_2;A,B\}$ is a regular medium, then it
is straightforward to verify
\begin{equation}\label{eq:equivalent}
\mathcal{E}_{[A,B]}(\phi)\sim \|\phi\|_{H^1(\mathbf{B}_2)^m}.
\end{equation}
Here and in the sequel, for two relations $\mathrm{R}_1,
\mathrm{R}_2$, $\mathrm{R}_1\sim\mathrm{R}_2$ means that there
exists two finite positive constants $c_1, c_2$ such that
$c_1\mathrm{R}_1\leq \mathrm{R}_2\leq c_2\mathrm{R}_1$. Also, in the
following, for notational convenience, we shall frequently refer to
$\mathrm{R}_1\lesssim \mathrm{R}_2$ as $\mathrm{R}_1\leq
c_2\mathrm{R}_2$. Next, one can verify directly that due to the
blow-up singularity of $A_e$ on $\mathbf{S}_1^+$,
\begin{equation}\label{eq:tangential condition}
\phi\in \mathscr{E}(\mathbf{B}_2)^m\ \ \mbox{and}\ \
\mathcal{E}_{[A_e,B_e]}(\phi)<\infty\ \ \mbox{if{f}}\ \ \
\frac{\partial\phi}{\partial\theta}\bigg|_{\mathbf{S}_1}=0.
\end{equation}
Here and also in what follows, we make use of the standard polar
coordinate $(y_1,y_2)\mapsto (r\cos\theta, r\sin\theta)$ in
$\mathbb{R}^2$. Hence, we set
\begin{equation}
\mathscr{T}^\infty(\mathbf{B}_2)^m:=\{\phi\in
\mathscr{E}(\mathbf{B}_2)^m;\
\frac{\partial\phi}{\partial\theta}\bigg|_{\mathbf{S}_1}=0\},
\end{equation}
and
\begin{equation}
\mathscr{T}_0^\infty(\mathbf{B}_2)^m:=\mathscr{T}^\infty(\mathbf{B}_2)^m\cap
\mathscr{D}(\mathbf{B}_2)^m,
\end{equation}
which are closed subspaces of $\mathscr{E}(\mathbf{B}_2)^m$. Set
\[
W_e(y):=\frac{4(|y|-1)}{|y|}I_m\quad\mbox{for $1<|y|<2$};\ \
I_m\quad \mbox{for $|y|\leq 1$}.
\]
We note by using the fact $B_b$ is a regular material parameter,
together with (\ref{eq:cloak2}) and (\ref{eq:algebraic2}), that for
$\phi\in \mathscr{T}^\infty(\mathbf{B}_2)$
\begin{equation}\label{eq:equivalent2}
\mathcal{E}_{[A_e, W_e]}(\phi)\sim\mathcal{E}_{[A_e, B_e]}(\phi).
\end{equation}
On the other hand, it is easy to see that
$\mathcal{E}_{[A_e,W_e]}(\cdot)$ defines a norm on
$\mathscr{T}^\infty(\mathbf{B}_2)^m$. Let
\begin{equation}\label{eq:weighted space}
H_{[A_e,B_e]}^1(\mathbf{B}_2)^m:=\mbox{cl}\{\mathscr{T}^\infty(\mathbf{B}_2)^m;
\mathcal{E}_{[A_e,W_e]}(\cdot)\},
\end{equation}
that is, the closure of the linear function space
$\mathscr{T}^\infty(\mathbf{B}_2)^m$ with respect to the singularly
weighted Sobolev norm $\mathcal{E}_{[A_e,W_e]}(\cdot)$. Clearly, one
can consider the elements in $H_{[A_e,B_e]}^1(\mathbf{B}_2)^m$ as
$\mathbb{C}^m$-valued measurable functions. Moreover, we have

\begin{lem}\label{lem:measurable}
The map
\[
\phi\mapsto D_{A_e}^\alpha\phi:=\sum_{\beta=1,2}
A_e^{\alpha\beta}\partial_\beta\phi,\
\phi\in\mathscr{T}^\infty(\mathbf{B}_2)^m,\ \alpha=1,2
\]
has a bounded extension
\begin{equation}\label{eq:extension}
D_{A_e}^\alpha:\
H^1_{[A_e,B_e]}(\mathbf{B}_2)^m\mapsto\mathcal{M}(\mathbf{B}_2;\mathbb{C}^m),
\end{equation}
where $\mathcal{M}(\mathbf{B}_2;\mathbb{C}^m)$ represents the space
of complex $\mathbb{C}^m$-valued Borel measures on $\mathbf{B}_2$.
Moreover, for $u\in H^1_{[A_e,B_e]}(\mathbf{B}_2)^m$, we have in the
sense of Borel measures
\begin{equation}\label{eq:borel}
(D_{A_e}^\alpha u)(\mathbf{S}_1)=0,\ \ \alpha=1,2.
\end{equation}

\end{lem}
\begin{proof}
Let $\phi\in \mathscr{T}^\infty(\mathbf{B}_2)^m$, $g\in
C(\mathbf{B}_2)^m$ and $\tilde\phi=F^*\phi, \tilde g=F^* g\in
L^\infty(B_2)^m$. Then, it is straightforward to show that
$D_{A_e}^\alpha\phi\in L^\infty(\mathbf{B}_2)^m$. Hence, it has that
\begin{align*}
&\int_{\mathbf{B}_2}(D_{A_e}^\alpha\phi)^* g\
dy=\int_{\mathbf{B}_2\backslash\mathbf{S}_1} (D_{A_e}^\alpha\phi)^*
g\ dy\\
=& \int_{\mathbf{B}_2\backslash\{0\}}\left(\sum_{k,l\leq
1}\frac{\partial y_\alpha}{\partial
x_k}\frac{\partial\tilde{\phi}}{\partial x_l}
A_b^{kl}\right)^*\tilde g\
dx+\int_{\mathbf{B}_1}(D_{A_a}^\alpha\phi)^* g\ dx
\end{align*}
On $\mathbf{B}_2$, one has $\partial y_\alpha/\partial
x_k=\mathcal{O}(\frac 1 r)$. This together with the facts that
$\{\mathbf{B}_2; A_b, B_b\}$ and $\{\mathbf{B}_1; A_a, B_a\}$ are
regular, we use (\ref{eq:equivalent}) to further have
\begin{align}
&|\int_{\mathbf{B}_2}(D_{A_e}^\alpha\phi)^* g\ dy|\nonumber\\
\lesssim &\ \|\tilde\phi\|_{H^1(\mathbf{B}_2)^m}\|\tilde
g/r\|_{L^2(\mathbf{B}_2)^m}+\|\phi\|_{H^1(\mathbf{B}_2)^m}\|g\|_{L^2(\mathbf{B}_1)^m}\nonumber\\
\lesssim &\
[\mathcal{E}_{[A_b,B_b]}(\tilde\phi)+\mathcal{E}_{[A_a,B_a]}(\phi)]\|g\|_{C(\mathbf{B}_2)^m}\mbox{dist}(\mbox{supp}(g), \mathbf{S}_1)\label{eq:01}\\
\lesssim &\ \mathcal{E}_{[A_e,
B_e]}(\phi)\|g\|_{C(\mathbf{B}_2)^m}\mbox{dist}(\mbox{supp}(g),
\mathbf{S}_1)\label{eq:02}\\
\lesssim &\ \mathcal{E}_{[A_e,
W_e]}(\phi)\|g\|_{C(\mathbf{B}_2)^m}\mbox{dist}(\mbox{supp}(g),
\mathbf{S}_1)\label{eq:03}\\
\lesssim &\
\|\phi\|_{H^1_{[A_e,B_e]}(\mathbf{B}_2)^m}\|g\|_{C(\mathbf{B}_2)^m}\mbox{dist}(\mbox{supp}(g),
\mathbf{S}_1).\nonumber
\end{align}
In the above inequalities, from (\ref{eq:01}) to (\ref{eq:02}), we
have made use the following facts by using Lemma~\ref{lem:trans opt}
\begin{equation}
\mathcal{E}_{[A_b,B_b]}(\tilde\phi)=\mathcal{Q}_{[A_b,-B_b]}(\tilde\phi,\tilde\phi)=\mathcal{Q}_{[F_*A_b,-F_*B_b]}(\phi,\phi)
=\mathcal{E}_{[A_e,B_e]}(\phi),
\end{equation}
whereas from (\ref{eq:02}) to (\ref{eq:03}), we have made use the
equivalence (\ref{eq:equivalent2}). This proves the bounded
extension (\ref{eq:extension}). Finally, (\ref{eq:borel}) follows by
taking functions $g$ supported in sufficiently small neighborhoods
of $\mathbf{S}_1$.

The proof is completed.
\end{proof}

Now, the solution to the singular system of PDEs (\ref{eq:cloaking
Helmholtz}) is defined by the distributional duality as to find
$u\in H^1_{[A_e,B_e]}(\mathbf{B}_2)^m$ such that
$u|_{\mathbf{S}_2}=h\in H^{1/2}(\mathbf{S}_2)^m$ and
\begin{equation}\label{eq:distributional equation}
\mathcal{Q}_{[A_e,B_e]}(u,\phi)=\langle f_e,\phi \rangle,\ \ \forall
\phi\in\mathscr{T}_0^\infty(\mathbf{B}_2)^m.
\end{equation}
\begin{rem}
Since the singularities of $A_e$ and $B_e$ are only attached to
$\mathbf{S}_1^+$, we know that for any $u\in
H_{[A_e,B_e]}^1(\mathbf{B}_2)^m$, $u\in
H^1_{loc}(\mathbf{B}_2\backslash \mathbf{S}_1)^m$. Therefore, for
(\ref{eq:distributional equation}) we have the well-defined
$u|_{\mathbf{S}_2}=h\in H^{1/2}(\mathbf{S}_2)^m$ and also a
well-defined Dirichlet-to-Neumann map on $\mathbf{S}_2$ defined by
\begin{equation}
\Lambda_{A_e,B_e, f_e}^\omega(h)=\sum_{\alpha,\beta\leq 1}\nu_\alpha
A_e^{\alpha\beta}\partial_\beta u\in H^{-1/2}(\mathbf{S}_2)^m,
\end{equation}
provided (\ref{eq:distributional equation}) has a unique solution.
\end{rem}


\begin{rem}
As is known,
\[
\bigg|\int_{\mathbf{B}_2}\left(\sum_{\alpha,\beta=1}^2(A_e^{\alpha\beta}\partial_\beta
u)^*\partial_\alpha u +(B_eu)^*u\right)\ dy\bigg|
\]
is the (generalized) energy of the system. Hence, the solution in
(\ref{eq:distributional equation}) is physically meaningful in that
it has finite energy.
\end{rem}

Next, we shall show the following on the solution of
(\ref{eq:distributional equation}).

\begin{thm}\label{thm:2} $u\in
H^1_{[A_e,B_e]}(\mathbf{B}_2)^m$ is a solution to
(\ref{eq:distributional equation}) if and only if $\tilde
v(x)=(F^*(u|_{\mathbf{B}_2\backslash\bar{\mathbf{B}}_1}))^e \in
H^1(\mathbf{B}_2)^m$ with
$(F^*(u|_{\mathbf{B}_2\backslash\bar{\mathbf{B}}_1}))^e$ denoting
the extension of
$F^*(u|_{\mathbf{B}_2\backslash\bar{\mathbf{B}}_1})$ from
$\mathbf{B}_2\backslash\{0\}$ to $\mathbf{B}_2$ (e.g., by setting it
be 0), is a solution to
\begin{equation}\label{eq:thm11}
\mathcal{P}_{[A_b,B_b]}\tilde v=f_b\ \ \mbox{on\ $\mathbf{B}_2$},\ \
\tilde v|_{\mathbf{S}_2}=h
\end{equation}
and $w=u|_{\mathbf{B}_1}\in H^1(\mathbf{B}_1)^m$ is a solution to
\begin{equation}\label{eq:thm12}
\mathcal{P}_{[A_a,B_a]}w=f_a\ \ \mbox{on\ $\mathbf{B}_1$},\ \
w|_{\mathbf{S}_1}=\mathbf{c}_0,
\end{equation}
where $\mathbf{c}_0\in\mathbb{C}^m$ is a constant vector determined
by
\begin{equation}\label{eq:thm13}
\int_{\mathbf{S}_1}\sum_{\alpha,\beta\leq 1}\nu_\alpha
A_a^{\alpha\beta}\partial_\beta w\ dS=0.
\end{equation}
\end{thm}
As a direct consequence of Theorem~\ref{thm:2}, we first give the
proof of Theorem~\ref{thm:1}.

\begin{proof}[Proof of Theorem~\ref{thm:1}]
Let $u\in H^1_{[A_e,B_e]}(\mathbf{B}_2)^m$ be a solution to
(\ref{eq:cloaking Helmholtz}) corresponding to the cloaking problem
for $\{\mathbf{B}_2;A_e,B_e, f_e\}$, whereas $\tilde v\in
H^1(\mathbf{B}_2)^m$ be a solution to (\ref{eq:thm11}) corresponding
to the differential equation in the reference space
$\{\mathbf{B}_2;A_b,B_b,f_b\}$. By Theorem~\ref{thm:2}, we have
\[
u|_{\mathbf{B}_2\backslash\bar{\mathbf{B}}_1}=(F^{-1})^*(\tilde
v|_{\mathbf{B}_2\backslash\{0\}}),\quad \tilde
v|_{\mathbf{B}_2\backslash\{0\}}=F^*(u|_{\mathbf{B}_2\backslash\bar{\mathbf{B}}_1}).
\]
Next, let $\kappa(t)\in\mathscr{D}(\mathbb{R})$ be a smooth real
cut-off function such that $0\leq \kappa(t)\leq 1$ with
$\kappa(t)=1$ for $t<4/3$ and $\kappa(t)=0$ for $t>5/3$. For
arbitrary $\phi\in \mathscr{E}(\mathbf{B}_2)$, set
$\psi(y):=(1-\kappa(|y|))\phi(y)\in
\mathscr{T}^\infty(\mathbf{B}_2)$. By Green's identity, we have
\begin{align}
&\mathcal{Q}_{[A_c,B_c]}(u,\psi)=\mathcal{Q}_{[A_e,B_e]}(u,\psi)\nonumber\\
=&\langle f_e,\psi\rangle_{\mathbf{B}_2}+\langle
\Lambda^\omega_{A_e,B_e,f_e}(h),\gamma\psi \rangle_{\mathbf{S}_2}\nonumber\\
=&\langle f_c,\psi
\rangle_{\mathbf{B}_2\backslash\bar{\mathbf{B}}_1}+\langle
\Lambda^\omega_{A_e,B_e,f_e}(h),\gamma\psi
\rangle_{\mathbf{S}_2}.\label{eq:001}
\end{align}
Then, by change of variables in integrations in (\ref{eq:001}), and
using Lemma~\ref{lem:trans opt} and,
$F|_{\mathbf{S}_2}=\mbox{Identity}$ and $f_c=J^{-1}(F^{-1})^*f_b$,
we further have
\begin{equation}\label{eq:002}
\mathcal{Q}_{[A_b,B_b]}(\tilde v, \tilde\psi)=\langle f_b,\tilde\psi
\rangle_{\mathbf{B}_2}+\langle
\Lambda^\omega_{A_e,B_e,f_e}(h),\gamma\psi \rangle_{\mathbf{S}_2},
\end{equation}
where $\tilde\psi=F^*\psi\in\mathscr{E}(\mathbf{B}_2)$. By using
Green's identity again, we know
\[
\mathcal{Q}_{[A_b,B_b]}(\tilde v, \tilde\psi)=\langle f_b,
\tilde\psi\rangle_{\mathbf{B}_2}+\langle
\Lambda^\omega_{A_b,B_b,f_b}(h),\gamma\psi \rangle_{\mathbf{S}_2},
\]
which implies by (\ref{eq:002}) that
\[
\langle \Lambda^\omega_{A_e,B_e,f_e}(h),\gamma\psi
\rangle_{\mathbf{S}_2}=\langle
\Lambda^\omega_{A_b,B_b,f_b}(h),\gamma\psi \rangle_{\mathbf{S}_2}
\]
and hence
\[
\Lambda^\omega_{A_e,B_e,f_e}=\Lambda^\omega_{A_b,B_b,f_b}.
\]

The proof is completed.

\end{proof}

We proceed to the proof of Theorem~\ref{thm:2}. We first derive two
auxiliary lemmata characterizing the singularly weighted Sobolev
space $H^1_{[A_e,B_e]}(\mathbf{B}_2)$. In the following,
$\chi_\Omega$ denotes the characteristic function for a set
$\Omega\subset\mathbb{R}^2$.

\begin{lem}\label{lem:1}
Let $\mathbf{c}_1$ and $\mathbf{c}_2$ be two constant vectors in
$\mathbb{C}^m$. Then
\begin{equation}
\mathbf{c}_1\chi_{\bar{\mathbf{B}}_1}+\mathbf{c}_2\chi_{\mathbf{B}_2\backslash
\bar{\mathbf{B}}_1}\in H^1_{[A_e,B_e]}(\mathbf{B}_2)^m.
\end{equation}
\end{lem}

\begin{proof}
Since a constant function always belongs to
$H^1_{[A_e,B_e]}(\mathbf{B}_2)^m$, it suffices to show that
$\varphi(y):=\chi_{\bar{\mathbf{B}}_1}\mathbf{c}\in
H^1_{[A_e,B_e]}(\mathbf{B}_2)^m$ with $\mathbf{c}\in\mathbb{C}^m$ a
constant vector. Let $\rho\in \mathscr{D}(\mathbb{R})$ be a cut-off
function such that $0\leq \rho(t)\leq 1$ with $\rho(t)=1$ for
$t<1/2$ and $\rho(t)=0$ for $t>1$. Then, define for $\varepsilon>0$,
\begin{equation}
\varphi_\varepsilon(y)=\begin{cases} & \mathbf{c},\qquad\qquad\qquad  y\in \bar{\mathbf{B}}_{1},\\
& \rho(\frac{\ln \varepsilon}{\ln(|y|-1)})\mathbf{c},\quad\ \ \ y\in
\mathbf{B}_{2}\backslash \bar{\mathbf{B}}_{1}.
\end{cases}
\end{equation}
Obviously, $\varphi_\varepsilon\in
\mathscr{T}^\infty(\mathbf{B}_2)^m$. Next, we shall show
\begin{equation}\label{eq:convergence}
\lim_{\varepsilon\rightarrow
0^+}\|\varphi_\varepsilon-\varphi\|_{H^1_{[A_e,B_e]}(\mathbf{B}_2)^m}=0,
\end{equation}
which then implies that $\varphi\in
H^1_{[A_e,B_e]}(\mathbf{B}_2)^m$. In fact, we have
\begin{equation}\label{eq:estimate1}
\begin{split}
&\|\varphi_\varepsilon-\varphi\|^2_{H^1_{[A_e,B_e]}(\mathbf{B}_2)^m}=\mathcal{E}^2_{[A_e,W_e]}(\varphi_\varepsilon-\varphi)\\
\lesssim&\
\mathcal{E}^2_{[A_e,B_e]}(\varphi_\varepsilon-\varphi)=\mathcal{E}_{[A_c,B_c]}^2(\varphi_\varepsilon)\\
=&\
\bigg|\int_{\mathbf{B}_2\backslash\bar{\mathbf{B}}_1}\left(\sum_{\alpha,\beta\leq
1}(A_c^{\alpha\beta}\partial_\beta\varphi_\varepsilon)^*\partial_\alpha\varphi_\varepsilon+(B_c\varphi_\varepsilon)^*\varphi_\varepsilon\right)\
dy \bigg|
\end{split}
\end{equation}
By using the explicit expression for $B_c$ in (\ref{eq:cloak2})
together with the fact the $B_b$ is a regular material parameter, it
is straightforwardly shown that
\begin{equation}\label{eq:ineq1}
\begin{split}
\left|\int_{\mathbf{B}_2\backslash\bar{\mathbf{B}}_1}
(B_c\varphi)^*\varphi\ dy\right|\lesssim &
\int_{\mathbf{B}_2\backslash\bar{\mathbf{B}}_1}\left|\rho(\frac{\ln\varepsilon}{\ln(|y|-1)})\right|^2\
dy\\
\lesssim &\ \mbox{Vol}(\Gamma_\varepsilon)\rightarrow 0\ \ \ \
\mbox{as\ \ $\varepsilon\rightarrow 0^+$},
\end{split}
\end{equation}
where
\[
\Gamma_\varepsilon:=\{y\in\mathbb{R}^2; 1<|y|<1+\varepsilon\}.
\]
Next, by direct calculations, we have
\begin{equation}\label{eq:gradient}
\nabla\left(\rho(\frac{\ln\varepsilon}{\ln
(|y|-1)})\right)=\rho'(\frac{\ln\varepsilon}{\ln
(|y|-1)})\cdot\frac{\ln\varepsilon}{|\ln(|y|-1)|^2}\cdot\frac{\hat{y}}{|y|-1},
\end{equation}
where $\hat{y}:=y/\|y\|\in \mathbf{S}_1$.
By using (\ref{eq:cloak1}) and (\ref{eq:gradient}), we further have
\begin{equation}\label{eq:expression1}
\begin{split}
&\left(\nabla\rho(\frac{\ln\varepsilon}{\ln
(|y|-1)})\otimes\mathbf{c}\right)^*A_c^*\left(\nabla\rho(\frac{\ln\varepsilon}{\ln
(|y|-1)})\otimes\mathbf{c}\right)\\
=& \left[\rho'(\frac{\ln\varepsilon}{\ln
(|y|-1)})\right]^2\cdot\frac{|\ln\varepsilon|^2}{|\ln(|y|-1)|^4}\cdot\frac{1}{|y|-1}\cdot[(A_b(\hat{y}\otimes\mathbf{c}))^*(\hat{y}\otimes\mathbf{c})]
\end{split}
\end{equation}
By using (\ref{eq:expression1}), we can deduce
\begin{equation}\label{eq:estimate2}
\begin{split}
&\left|\int_{\mathbf{B}_2\backslash\bar{\mathbf{B}}_1}\sum_{\alpha,\beta\leq
1}
(A_c^{\alpha\beta}\partial_\beta\varphi_\varepsilon)^*\partial_\alpha\varphi_\varepsilon
\ dy\right|\\
=& \bigg|\int_{\mathbf{B}_2\backslash\bar{B}_1}
\left(\nabla\rho(\frac{\ln\varepsilon}{\ln
(|y|-1)})\otimes\mathbf{c}\right)^*A_c^*\left(\nabla\rho(\frac{\ln\varepsilon}{\ln
(|y|-1)})\otimes\mathbf{c}\right)\ dy\bigg| \\
\lesssim &
\left|\int_{\Gamma_\varepsilon}\frac{|\ln\varepsilon|^2}{|\ln(|y|-1)|^4}\cdot\frac{1}{|y|-1}\
dy\right|\rightarrow 0\ \ \ \mbox{as\ \ $\varepsilon\rightarrow
0^+$}.
\end{split}
\end{equation}
Combining (\ref{eq:estimate1}),(\ref{eq:ineq1}) and
(\ref{eq:estimate2}), we have (\ref{eq:convergence}). The proof is
completed.
\end{proof}

\begin{lem}\label{lem:characterization}
Let $u$ be a measurable $\mathbb{C}^m$-valued function on
$\mathbf{B}_2$. Then $u\in H_{[A_e,B_e]}^1(\mathbf{B}_2)^m$ if and
only if the following two conditions hold:
\begin{enumerate}
\item[(i)] $w:=u|_{\mathbf{B}_1}\in H^1(\mathbf{B}_1)^m$ and
\begin{equation}
\gamma^-w:=w|_{\mathbf{S}_1^-}=\mbox{constant},
\end{equation}
where $\gamma^-$ is the trace operator on $\mathbf{S}_1^-$, i.e. as
one approaches $\mathbf{S}_1$ from $\mathbf{B}_1$.

\item[(ii)] $v:=u|_{\mathbf{B}_2\backslash\bar{\mathbf{B}}_1}$ satisfies
\begin{equation}\label{eq:thm22}
\left|\int_{\mathbf{B}_2\backslash\bar{\mathbf{B}}_1}\sum_{\alpha,\beta\leq
1}(A_c^{\alpha\beta}\partial_\beta v)^*\partial_\alpha v\
dy\right|<\infty, \ \
\left|\int_{\mathbf{B}_2\backslash\bar{\mathbf{B}}_1} (B_c v)^* v\
dy\right|<\infty
\end{equation}
and
\begin{equation}\label{eq:thm23}
e_\theta\cdot\nabla v_p|_{\mathbf{S}_1^+}=0,\ \ \ p=1,2,\ldots,m,
\end{equation}
where $e_\theta$ is the unit angular directional vector on the
sphere $\mathbf{S}_1$.
\end{enumerate}
\end{lem}
\begin{proof}
First, we show conditions (i) and (ii) are necessary for  $u\in
H^1_{[A_e,B_e]}(\mathbf{B}_2)^m$.

Clearly, we have
\[
\mathcal{E}_{[A_a,B_a]}(w)\lesssim\|u\|_{H^1_{[A_e,B_e]}(\mathbf{B}_2)^m},
\]
which together with the fact $\{\mathbf{B}_1;A_a,B_a\}$ is regular
implies that
\[
\|w\|_{H^1(\mathbf{B}_1)^m}\lesssim
\|u\|_{H^1_{[A_e,B_e]}(\mathbf{B}_2)^m},\quad \mbox{i.e.,}\quad w\in
H^1(\mathbf{B}_1)^m.
\]
Let $\{\phi_n\}_{n=1}^\infty\subset
\mathscr{T}^\infty(\mathbf{B}_2)^m$ be such that
\[
\|\phi_n-u\|_{H^1_{[A_e,B_e]}(\mathbf{B}_2)^m}\rightarrow 0\quad
\mbox{as $n\rightarrow\infty$}.
\]
Then, we obviously have
\[
\|\phi_n|_{\mathbf{B}_1}-w\|_{H^1(\mathbf{B}_1)^m}\rightarrow 0\quad
\mbox{as $n\rightarrow 0$}.
\]
Hence,
\[
\|\gamma^-\phi_n-\gamma^- w\|_{H^{1/2}(\mathbf{S}_1)^m}\lesssim
\|\phi_n|_{\mathbf{B}_1}-w\|_{H^1(\mathbf{B}_1)^m}\rightarrow 0\quad
\mbox{as $n\rightarrow 0$}.
\]
Noting $\partial\phi_n/\partial\theta|_{\mathbf{S}_1}=0$, we know
$\gamma^-\phi_n$ are constants independent of the angular variable
$\theta$. Therefore, $\gamma^-w$ is constant. Whereas for
$v:=u|_{\mathbf{B}_2\backslash\bar{\mathbf{B}}_1}$, it is trivial to
see that (\ref{eq:thm22}) holds. To see (\ref{eq:thm23}), we let
$\tilde{v} = F^* v$. Using the change of variables in
(\ref{eq:thm22}), we have
\[
|\int_{\mathbf{B}_2\backslash\{0\}}\sum_{\alpha,\beta\leq
1}(A_b^{\alpha\beta}\partial_\beta)^*\partial_\alpha\tilde v\
dx|<\infty\ \ \ \mbox{and}\ \ \
|\int_{\mathbf{B}_2\backslash\{0\}}(B_b\tilde v)^*\tilde v\
dx|<\infty.
\]
Since the reference space $\{ \mathbf{B}_2; A_b, B_b \}$ is regular,
we further have by using (\ref{eq:algebraic}) and
(\ref{eq:algebraic2})
\[
\int_{\mathbf{B}_2\backslash\{0\}}\sum_{\alpha\leq
1}|\partial_\alpha\tilde v|^2\ dx<\infty\ \ \mbox{and}\ \
\int_{\mathbf{B}_2\backslash\{0\}}|\tilde v|^2\ dx<\infty.
\]
Set $\Psi_p=\nabla\tilde{v}_p|_{\mathbf{B}_2\backslash\{0\}},\
p=1,2,\ldots,m $. We have $\Psi_p\in
L^2(\mathbf{B}_2\backslash\{0\})^{2\times 1}$. Extending $\Psi$ and
$\tilde v$ on $\{0\}$ (e.g., by setting to be zero), and using the
same notations for the extensions, we have
$$
\Psi_p \in (L^2(\mathbf{B}_2))^{2\times 1} \mbox{ and } \tilde{v}_p
\in L^2(\mathbf{B}_2) .
$$
The difference $\nabla \tilde{v}_p - \Psi_p$ belongs to
$H^{-1}(\mathbf{B}_2)$ and it is supported on $\{0\}$. Since no
non-zero distribution in $H^{-1}(\mathbf{B}_2)$ is supported on
$\{0\}$, we see $\nabla\tilde{v}_p=\Psi_p$. Hence, $\tilde{v}_p\in
H^1(\mathbf{B}_2)$ and therefore $\tilde v\in H^1(\mathbf{B})^m$.


Let the Fourier decomposition of $\tilde v_p,\ 1\leq p\leq m,$ be
given by
\begin{equation} \label{eqn:Fourier_decompo}
\tilde v_p(r,\theta)=\frac{1}{\sqrt{2\pi}}\sum_{n
\in\mathbb{Z}}\tilde{v}_{p}^{(n)}(r)e^{in\theta}.
\end{equation}
In \cite{Bernardi1999aa}, it is proved
\begin{equation}
\label{eqn:Trace_Fourier_Coeff} \tilde{v}_p^{(n)}(0)=0,\ \ n \neq 0.
\end{equation}
Formally, we write $\tilde{v}_p(0, \theta) = \tilde{v}_p^{(0)}(0)$.
Since $v_p \chi_{B_2 \backslash \overline{B_1}} =(F^{-1})^*\tilde
v_p$, we have
\begin{equation*}
\forall \ 1 < r < 2,\ v_p(r,\theta) = \frac{1}{\sqrt{2\pi}}\sum_{n
\in\mathbb{Z}} v_p^{(n)}(r)e^{in\theta}\ \ \mbox{with}\ \
v_p^{(n)}(r) = \tilde{v}_p^{(n)}(2(r-1)) .
\end{equation*}
By (\ref{eqn:Trace_Fourier_Coeff}), we have
\[
v_p^{(n)}(1) = \tilde{v}_p^{(n)}(0)=0 \ \ \mbox{for } n \neq 0 ,
\]
which  implies
\begin{equation}\label{eqn:zero_trace}
\frac{1}{r}\frac{\partial
v_p}{\partial\theta}(r,\theta)\bigg|_{r=1}=0.
\end{equation}

Next, we show that (i) and (ii) are also sufficient conditions for a
measurable function $u$ on $\mathbf{B}_2$ to belong to
$H^1_{[A_e,B_e]}(\mathbf{B}_2)^m$. We first assume that
$u|_{\mathbf{B}_2\backslash\mathbf{B}_1}=0$. By (i), let
$\mathbf{c}_0$ be a constant vector such that
$u|_{\mathbf{S}_1^-}=\mathbf{c}_0$. According to Lemma~\ref{lem:1},
it suffices to show that
$\hat{u}:=u-\mathbf{c}_0\chi_{{\mathbf{B}}_1}\in
H^1_{[A_e,B_e]}(\mathbf{B}_1)^m$. Clearly,
\[
\hat{w}:=\hat{u}\chi_{\mathbf{B}_1}=w-\mathbf{c}_0\chi_{{\mathbf{B}}_1}\in
H_0^1(\mathbf{B}_1)^m.
\]
Hence, there exists $\{\varphi_n\}_{n=1}^\infty\subset \
\mathscr{D}(\mathbf{B}_1)^m$ such that
\[
\|\varphi_n-\hat{w}\|_{H^1(\mathbf{B}_1)^m}\rightarrow 0\quad
\mbox{as $n\rightarrow \infty$}.
\]
Let $\phi_n\in \mathscr{T}^\infty(\mathbf{B}_2)^m$ be such that
$\phi_n|_{\mathbf{B}_1}=\varphi_n$ and
$\phi_n|_{\mathbf{B}_2\backslash \mathbf{B}_1}=0$. Using the fact
that $\{\mathbf{B}_1;A_a,B_a\}$ is regular and
$\hat{u}|_{\mathbf{B}_2\backslash\bar{\mathbf{B}}_1}=0$, we have
\begin{equation*}
\begin{split}
\|\phi_n-\hat{u}\|_{H^1_{[A_e,B_e]}(\mathbf{B}_2)^m}^2=& \mathcal{E}^2_{[A_a,B_a]}(\varphi_n-\hat{w})\\
\lesssim & \
\|\varphi_n-\hat{w}\|_{H^1(\mathbf{B}_1)^m}^2\rightarrow
0\quad\mbox{as $n\rightarrow\infty$},
\end{split}
\end{equation*}
which implies that $\hat{u}\in H^1_{(A_e,B_e)}(\mathbf{B}_1)^m$.

Now, let $u$ be a measurable function satisfying (i) and (ii). As is
shown above, $u\chi_{\mathbf{B}_1}\in
H^1_{[A_e,B_e]}(\mathbf{B}_2)^m$, one only needs to show that
$u-u\chi_{\mathbf{B}_1}\in H^1_{[A_e,B_e]}(\mathbf{B}_2)^m$; that
is, $u$ vanishes inside $\mathbf{B}_1$. By (\ref{eq:thm23}), we know
$u|_{\mathbf{S}_1^+}$ is constant independent of the angular
variable $\theta$. Using similar argument as earlier by substraction
of a Heaviside function from $u$ together with Lemma~\ref{lem:1}, we
can further assume that $u$ vanishes on $\bar{\mathbf{B}}_1$.
Moreover, without loss of generality, we can also assume that $u$
vanishes near $\mathbf{S}_2$. Let $\tilde
v:=F*(u|_{\mathbf{B}_2\backslash\bar{\mathbf{B}}_1})=F^*v$.
According to our earlier argument, $v\in H^1(\mathbf{B}_2)^m$ (note
here we identify $v$ and its extension on $\{0\}$). Since $v$
vanishes near $\mathbf{S}_2$ and $\{0\}$ is a $(2,1)$-polar set,
there are $\tilde\phi_n\in
\mathscr{D}(\mathbf{B}_2\backslash\{0\})^m$ such that
\[
\|\tilde\phi_n-\tilde{v}\|_{H^1(\mathbf{B}_2)^m}\rightarrow
0\quad\mbox{as $n\rightarrow\infty$}.
\]
Let $\phi_n\in \mathscr{T}^\infty(\bar{\mathbf{B}}_2)^m$ be such
that
\[
\phi_n|_{\mathbf{B}_2\backslash\bar{\mathbf{B}}_1}=(F^{-1})^*\tilde{\phi}_n\quad\mbox{and}\quad
\phi_n|_{\bar{\mathbf{B}}_1}=0.
\]
Using $u|_{\bar{\mathbf{B}}_1}=0$ and the fact that
$\{\mathbf{B}_2;A_b,B_b\}$ is regular, we have
\begin{align*}
\|\phi_n-u\|_{H^1_{[A_e,B_e]}(\mathbf{B}_2)^m}\lesssim\ & \mathcal{E}_{[\mathbf{B}_2\backslash\bar{\mathbf{B}}_1; A_c, B_c]}(\phi_n-v)\\
=\ & \mathcal{E}_{[A_b,B_b]}(\tilde\phi_n-\tilde v)\\
\lesssim\ & \|\tilde\phi_n-\tilde
v\|_{H^1(\mathbf{B}_2)^m}\rightarrow 0\quad\mbox{as
$n\rightarrow\infty$},
\end{align*}
which implies that $u\in H^1_{[A_e,B_e]}(\mathbf{B}_2)^m$.

The proof is completed.
\end{proof}

Now, we are ready to present the proof of Theorem~\ref{thm:2}.

\begin{proof}[Proof of Theorem~\ref{thm:2}]
We first show that if $u\in H^1_{[A_e,B_e]}(\mathbf{B}_2)^m$ is a
solution to (\ref{eq:cloaking Helmholtz}) and
(\ref{eq:distributional equation}), then we must have the decoupled
problems (\ref{eq:thm11}) and (\ref{eq:thm12}).

Let $v:=u|_{\mathbf{B}_2\backslash\bar{\mathbf{B}}_1}$ and $\tilde
v=F^*\tilde{v}$. In the following, as in the proof of
Lemma~\ref{lem:characterization}, we identify $\tilde v$ and its
$H^1$-extension from $\mathbf{B}_2\backslash\{0\}$ to
$\mathbf{B}_2$. Clearly, $\tilde v$ satisfies (\ref{eq:thm11}). On
the other hand, let $w:=u|_{\mathbf{B}_1}$, then $\gamma^-w$ is a
constant vector independent of the angular variable $\theta$ by
Lemma~\ref{lem:characterization}. So, it is trivial to see that $w$
satisfies (\ref{eq:thm12}). Hence, we only need to show
(\ref{eq:thm13}), whereas the determination of $\mathbf{c}_0$ from
(\ref{eq:thm13}) will be discussed in the subsequent
Theorem~\ref{thm:wellposedness}.

Set
\[
\Sigma_\varepsilon^+:=\{y\in\mathbb{R}^2; 1<|y|<1+\varepsilon/2\}\ ,
\ \Sigma_\varepsilon^-:=\{y\in\mathbb{R}^2; 1-\varepsilon/2<|y|<1\},
\]
and
\[
\Upsilon_\varepsilon^+:=\{y\in\mathbb{R}^2; |y|=1+\varepsilon/2\}\ ,
 \ \Upsilon_\varepsilon^-:=\{y\in\mathbb{R}^2;
|y|=1-\varepsilon/2\}.
\]
We first show for $\varphi\in\mathscr{T}_0^\infty(\mathbf{B}_2)^m$
\begin{equation}\label{eq:a1}
\lim_{\varepsilon\rightarrow
0^+}\left|\int_{\Sigma_\varepsilon^+}\left(\sum_{\alpha,\beta\leq
1}(A_e^{\alpha\beta}\partial_\beta
u)^*\partial_\alpha\varphi-\omega^2(B_e u)^*\varphi-f_e^*\varphi\
\right)dy\right|=0.
\end{equation}
In fact, by using the expression of $(A_e,B_e)$ in
$\Sigma_\varepsilon^+$ and the change of variables in integrations,
together with $f_e$ being $J^{-1}(F^{-1})^*f_b$ in
$\Sigma^+_\varepsilon$ and $\tilde{v}$ being the $H^1$-extension of
$F^*v$ from $\mathbf{B}_2\backslash\{0\}$ to $\mathbf{B}_2$, we have
\begin{equation}\label{eq:e01}
\begin{split}
&\left|\int_{\Sigma_\varepsilon^+}\left(\sum_{\alpha,\beta\leq
1}(A_e^{\alpha\beta}\partial_\beta
u)^*\partial_\alpha\varphi-\omega^2(B_e u)^*\varphi-f_e^*\varphi\
\right)dy\right|\\
=\ & \left|\int_{\mathbf{B}_\varepsilon}\left(\sum_{\alpha,\beta\leq
1}(A_b^{\alpha\beta}\partial_\beta \tilde
v)^*\partial_\alpha\tilde\varphi-\omega^2(B_b \tilde
v)^*\tilde\varphi-f_b^*\tilde \varphi\ \right)dx\right|\\
\lesssim\ & \left(\|\tilde
v\|_{H^1(\mathbf{B}_\varepsilon)^m}+\|f_b\|_{\widetilde{H}^{-1}(\mathbf{B}_\varepsilon)^m}\right)\|\tilde\varphi\|_{H^1(\mathbf{B}_\varepsilon)^m}
\rightarrow 0 \quad\mbox{as $\varepsilon\rightarrow 0^+$},
\end{split}
\end{equation}
where $\tilde \phi\in H^1(\mathbf{B}_2)^m$ is the continuous
extension of $F^*(\phi|_{\mathbf{B}_2\backslash
\bar{\mathbf{B}}_1})$ from $\mathbf{B}_2\backslash\{0\}$ to
$\mathbf{B}_2$. On the other hand, noting $u|_{\mathbf{B}_1}=w\in
H^1(\mathbf{B}_1)^m$ and $f_a|_{\mathbf{B}_1}\in
\widetilde{H}^{-1}(\mathbf{B}_1)^m$, it is straightforward to verify
\begin{equation}\label{eq:e02}
\begin{split}
&\lim_{\varepsilon\rightarrow
0^+}\left|\int_{\Sigma_\varepsilon^-}\left(\sum_{\alpha,\beta\leq
1}(A_e^{\alpha\beta}\partial_\beta
u)^*\partial_\alpha\varphi-\omega^2(B_e u)^*\varphi-f_e^*\varphi\
\right)dy\right|\\
=& \lim_{\varepsilon\rightarrow
0^+}\left|\int_{\Sigma_\varepsilon^-}\left(\sum_{\alpha,\beta\leq
1}(A_a^{\alpha\beta}\partial_\beta
w)^*\partial_\alpha\varphi-\omega^2(B_e w)^*\varphi-f_a^*\varphi\
\right)dy\right|\\
=&\ 0.
\end{split}
\end{equation}

Next, by (\ref{eq:distributional equation}) and using (\ref{eq:e01})
and (\ref{eq:e02}), together with integration by parts, we have
\begin{equation}\label{eq:e03}
\begin{split}
0=&\int_{\mathbf{B}_2}\left(\sum_{\alpha,\beta\leq
1}(A_e^{\alpha\beta}\partial_\beta
u)^*\partial_\alpha\varphi-\omega^2(B_e\varphi)^*\varphi-f_e^*\phi\right)\
dy\\
=&
\int_{\mathbf{B}_2\backslash\mathbf{S}_1}\left(\sum_{\alpha,\beta\leq
1}(A_e^{\alpha\beta}\partial_\beta
u)^*\partial_\alpha\varphi-\omega^2(B_e\varphi)^*\varphi-f_e^*\phi\right)\
dy\\
=&\lim_{\varepsilon\rightarrow
0^+}\int_{\mathbf{B}_2\backslash(\Sigma_\varepsilon^+\cup\mathbf{S}_1\cup\Sigma_\varepsilon^-)}
\left(\sum_{\alpha,\beta\leq 1}(A_e^{\alpha\beta}\partial_\beta
u)^*\partial_\alpha\varphi-\omega^2(B_e\varphi)^*\varphi-f_e^*\phi\right)
dy\\
=&\lim_{\varepsilon\rightarrow
0^+}\left(\int_{\Upsilon_\varepsilon^+}-(\sum_{\alpha,\beta\leq
1}\nu_\alpha A_c^{\alpha\beta}\partial_\beta u)^*\varphi
dS(y)+\int_{\Upsilon_\varepsilon^-}-(\sum_{\alpha,\beta\leq
1}\nu_\alpha A_a^{\alpha\beta}\partial_\beta u)^*\varphi
dS(y)\right)
\end{split}
\end{equation}
where, by a bit abuse of notations, $\nu$ denotes the exterior unit
normal to respective domain, $\Sigma_\varepsilon^+$ and
$\Sigma_\varepsilon^-$. Next, we estimate the integral in the right
hand side of the last equation in (\ref{eq:e03}). In the sequel, we
denote by $\tilde \nu$ the exterior unit normal vector to the domain
$B_\varepsilon$. Using the change of variables in integration and
the fact that
$\tilde{v}=F^*(u|_{\mathbf{B}_2\backslash\bar{\mathbf{B}}_1})$
belongs to $H^1(\mathbf{B}_2)^m$ and satisfies (\ref{eq:thm11}), we
have
\begin{equation}\label{eq:e04}
\begin{split}
&\left|\int_{\Upsilon_\varepsilon^+}(\sum_{\alpha,\beta\leq
1}\nu_\alpha A_c^{\alpha\beta}\partial_\beta u)^*\varphi
dS(y)\right|\\
=&\left|\int_{\partial\mathbf{B}_\varepsilon}(\sum_{\alpha,\beta\leq
1}\tilde{\nu}_\alpha A_b^{\alpha\beta}\partial_\beta \tilde
v)^*\tilde\varphi dS(x)\right|\\
=&\left|\int_{\mathbf{B}_\varepsilon}\left(\sum_{\alpha,\beta\leq
1}(A_b^{\alpha\beta}\partial_\beta\tilde
v)^*\partial_\alpha\tilde\varphi-\omega^2(B_b\tilde
v)^*\tilde\varphi-f_b^*\tilde\varphi\right)\
dx\right|\\
\lesssim& \left(\|\tilde
v\|_{H^1(\mathbf{B}_\varepsilon)^m}+\|f_b\|_{\widetilde{H}^{-1}(\mathbf{B}_\varepsilon)^m}\right)
\|\tilde\varphi\|_{H^1(\mathbf{B}_\varepsilon)^m}\rightarrow
0\quad\mbox{as\ $\varepsilon\rightarrow 0^+$}.
\end{split}
\end{equation}
By (\ref{eq:e03}) and (\ref{eq:e04}), we see
\[
\int_{\mathbf{S}_1}(\sum_{\alpha,\beta\leq 1}\nu_\alpha
A_a^{\alpha\beta}\partial_\beta w)^*\varphi
dS(y)=\lim_{\varepsilon\rightarrow 0^+}
\int_{\Upsilon_\varepsilon^-}(\sum_{\alpha,\beta\leq 1}\nu_\alpha
A_a^{\alpha\beta}\partial_\beta w)^*\varphi dS(y)=0
\]
which implies by the fact $\varphi|_{\mathbf{S}_1}$ could be an
arbitrary constant vector from $\mathbb{C}^m$ that
\[
\int_{\mathbf{S}_1}\sum_{\alpha,\beta\leq 1}\nu_\alpha
A_a^{\alpha\beta}\partial_\beta w \ dS(y)=0.
\]

Now, let $\tilde v\in H^1(\mathbf{B}_2)^m$ and $w\in
H^1(\mathbf{B}_1)^m$ be solutions, respectively to (\ref{eq:thm11})
and (\ref{eq:thm12}). Set $u$ be $(F^{-1})^*\tilde v$ on
$\mathbf{B}_2\backslash \bar{\mathbf{B}}_1$ and be $w$ on
$\mathbf{B}_{1}$, and extend it to $B_{2}$ by setting it be zero on
$\mathbf{S}_1$. By Lemma~\ref{lem:characterization}, we see that
$u\in H_{[A_e,B_e]}^1(\mathbf{B}_2)^m$. Moreover, it is readily seen
that one also has
\begin{equation}\label{eq:en1}
\mathcal{P}_{[A_c,B_c]} u=f_c\ \ \mbox{on
$\mathbf{B}_2\backslash\bar{\mathbf{B}}_1$},\ \ \
u|_{\mathbf{S}_2}=h,
\end{equation}
and
\begin{equation}\label{eq:en2}
\mathcal{P}_{[A_a,B_a]}u=f_a\ \ \mbox{on $\mathbf{B}_1$}.
\end{equation}
By Lemma~\ref{lem:measurable}, together with (\ref{eq:en1}) and
(\ref{eq:en2}), we have for any
$\varphi\in\mathscr{T}_0^\infty(\mathbf{B}_2)^m$
\begin{equation}\label{eq:en}
\begin{split}
&\mathcal{Q}_{[A_e,B_e]}(u,\varphi)-\langle
f_e,\varphi\rangle_\Omega\\
=&\lim_{\varepsilon\rightarrow
0^+}\int_{\mathbf{B}_2\backslash(\Sigma_\varepsilon^+\cup\mathbf{S}_1\cup\Sigma_\varepsilon^-)}
\left(\sum_{\alpha,\beta\leq 1}(A_e^{\alpha\beta}\partial_\beta
u)^*\partial_\alpha\varphi-\omega^2(B_e\varphi)^*\varphi-f_e^*\phi\right)
dy\\
=&\lim_{\varepsilon\rightarrow
0^+}\left(\int_{\Upsilon_\varepsilon^+}-(\sum_{\alpha,\beta\leq
1}\nu_\alpha A_c^{\alpha\beta}\partial_\beta u)^*\varphi
dS(y)+\int_{\Upsilon_\varepsilon^-}-(\sum_{\alpha,\beta\leq
1}\nu_\alpha A_a^{\alpha\beta}\partial_\beta u)^*\varphi
dS(y)\right)
\end{split}
\end{equation}

Again, using the estimate in (\ref{eq:e04}), we know
\begin{equation}\label{eq:en3}
\lim_{\varepsilon\rightarrow
0^+}\int_{\Upsilon_\varepsilon^+}(\sum_{\alpha,\beta\leq
1}\nu_\alpha A_c^{\alpha\beta}\partial_\beta u)^*\varphi\ dS(y)=0
\end{equation}
On the other hand, by noting $u=w$ on $\mathbf{B}_1$, and using
(\ref{eq:thm13}) and $\varphi|_{\mathbf{S}_1}$ is a constant vector
in $\mathbb{C}^m$, we further have
\begin{equation}\label{eq:en4}
\lim_{\varepsilon\rightarrow
0^+}\int_{\Upsilon_\varepsilon^-}(\sum_{\alpha,\beta\leq
1}\nu_\alpha A_a^{\alpha\beta}\partial_\beta u)^*\varphi\ dS(y)=0
\end{equation}
Finally, by (\ref{eq:en})--(\ref{eq:en4}), we have
\[
\mathcal{Q}_{[A_e,B_e]}(u,\varphi)=\langle f_e,\varphi
\rangle_\Omega,\ \ \forall
\varphi\in\mathscr{T}_0^\infty(\mathbf{B}_2).
\]
That is, $u\in H^1_{[A_e,B_e]}(\mathbf{B}_2)^m$ is a solution to
(\ref{eq:distributional equation}).

The proof is complete.
\end{proof}

In the rest of this section, we study the interior
problem~(\ref{eq:thm12})--(\ref{eq:thm13}). To that end, we
introduce the following closed subspace of $H^1(\mathbf{B}_1)^m$,
\begin{equation}\label{eq:subspace}
\mathscr{W}:=\{g\in H^1(\mathbf{B}_1)^m;\ \gamma
g|_{\mathbf{S}_1}=\mbox{constant}\}.
\end{equation}
Then, (\ref{eq:thm12})--(\ref{eq:thm13}) is weakly formulated as
\begin{equation}\label{eq:eigenvalue}
w\in\mathscr{W}\ \ \mbox{and}\ \
\mathcal{Q}_{[A_a,B_a]}(w,g)=\langle f_a, g \rangle_{\mathbf{B}_1},\
\ \forall g\in\mathscr{W}.
\end{equation}
The corresponding homogeneous problem is
\begin{equation}\label{eq:eigenvalue1}
w\in\mathscr{W}\ \ \mbox{and}\ \ \mathcal{Q}_{[A_a,B_a]}(w,g)=0,\ \
\forall g\in\mathscr{W},
\end{equation}
and by noting (\ref{eq:symmetric}), its adjoint problem is
\begin{equation}\label{eq:adjoint}
v\in\mathscr{W}\ \ \mbox{and}\ \ \mathcal{Q}_{[A,B^*]}(v,g)=0,\ \
\forall g\in\mathscr{W}.
\end{equation}

\begin{thm}\label{thm:wellposedness}
Let $W$ denote the set of solutions to (\ref{eq:eigenvalue1}). Then,
either (i)~$W=\{0\}$; or (ii)~$\mbox{dim}\ W=n$ for some finite
$n\geq 1$. In the case (i), the problem~(\ref{eq:eigenvalue1}) is
uniquely solvable. Whereas in case (ii), the homogeneous adjoint
problem (\ref{eq:adjoint}) also has exactly $n$ linearly independent
solutions, say $v_1, v_2,\ldots, v_n\in \mathscr{W}$, and the
inhomogeneous problem (\ref{eq:eigenvalue}) is solvable if{f}
\begin{equation}\label{eq:compatibility}
\langle v_p, f_a\rangle=0\ \ \mbox{for\ $1\leq p\leq n$}.
\end{equation}
\end{thm}

\begin{proof}
Consider the operator
$\mathcal{L}:\mathscr{W}\rightarrow\mathscr{W}^*$ determined by
$\mathcal{Q}_{[A_a,B_a]}$ in the standard way as
\[
(\mathcal{L}g_1)(g_2)=\mathcal{Q}_{[A_a,B_a]}(g_1,g_2).
\]
Let $\mathscr{H}=L^2(\mathbf{B}_1)^m$ act as the pivot space.
Clearly, the inclusion $\mathscr{W}\subseteq\mathscr{H}$ is compact.
So, $\mathcal{L}$ is Fredholm with index 0. Furthermore, each
distribution $f_a\in\widetilde{H}^{-1}(\mathbf{B}_1)^m$ gives rise
to a unique functional $\mathrm{F}_a\in\mathscr{W}^*$, defined by
$\mathrm{F}_a(g)=\langle f_a, g \rangle_{\mathbf{B}_1}$ for $g\in
\mathscr{W}$. Thus, the equation (\ref{eq:eigenvalue}) is equivalent
to
\begin{equation}\label{eq:fredholm}
\mathcal{L}w=\mathrm{F}_a.
\end{equation}
The celebrated Fredholm theory applied to (\ref{eq:fredholm}) gives
the desired results in the theorem.
\end{proof}

\begin{rem}
In order to guarantee the existence of solutions to
(\ref{eq:cloaking Helmholtz}), we impose the following compatibility
condition on the interior source/sink $f_a$,
\begin{equation}
\langle v,f_a \rangle=0,
\end{equation}
where $v$ is any solution to (\ref{eq:adjoint}).
\end{rem}

Concerning the solution to the cloaking problem (\ref{eq:cloaking
Helmholtz}), we have the following qualitative observations
\begin{rem}\label{rem:necessity}
Let $u\in H^1_{[A_e,B_e]}(\mathbf{B}_2)^m$ be a solution to
(\ref{eq:cloaking Helmholtz}). Set $\gamma^+$ and $\gamma^-$ be the
two sided trace operators on the cloaking $\mathbf{S}_1$. Then, by
(\ref{eq:e03}), we see
\begin{align*}
&\int_{\mathbf{S}_1}[\sum_{\alpha,\beta\leq 1}\nu_\alpha
A_e^{\alpha\beta}\partial_\beta u]\
dS\\
=&\int_{\mathbf{S}_1}[\gamma^+(\sum_{\alpha,\beta\leq 1}\nu_\alpha
A_e^{\alpha\beta}\partial_\beta u)-\gamma^-(\sum_{\alpha,\beta\leq
1}\nu_\alpha A_e^{\alpha\beta}\partial_\beta u)]\ dS\\
=&\ 0
\end{align*}
On the other hand, we see that generically one has
\[
[u]|_{\mathbf{S}_1}:=\gamma^+ u-\gamma^- u=\mbox{constant
vector}\neq 0,
\]
since otherwise we would have an over-determined system
(\ref{eq:thm12})--(\ref{eq:thm13}). This observation also
encompasses the necessity of introducing the finite energy solutions
other than spatial $H^1$-solutions to the singular cloaking problem
(\ref{eq:cloaking Helmholtz}).
\end{rem}

\section{General invisibility cloaking}

Up till now, we have considered invisibility cloaking where the
cloaked region is the unit disc $\mathbf{B}_1$, and the cloaking
layer is the annulus $\mathbf{B}_2\backslash\bar{\mathbf{B}}_1$. In
this section, we extend our study to the general cloaking. To that
end, let $G:\mathbf{B}_2\rightarrow\Omega$ be an
orientation-preserving and bi-Lipschitz mapping and let
$D=G(\mathbf{B}_1)$. Then, for $F$ in (\ref{eqn:F}), let
\[
\mathbf{K}=G\circ F\circ G^{-1}:\ \Omega\rightarrow \Omega,
\]
which blows up the point $G(\{0\})$ to $D$ within $\Omega$ while
keeps $\partial\Omega$ fixed. Now, we shall show

\begin{thm}\label{thm:general}
Let $\{\Omega; \mathbf{A_b, B_b, f_b}\}$ with $\mathbf{f}_b$
supported away from $G(\{0\})$ be the regular reference/background
space. Then
\[
\{\Omega\backslash\bar{D}; \mathbf{A_c, B_c,
f_c}\}=\mathbf{K}_*\{\Omega\backslash G(\{0\}); \mathbf{A_b, B_b,
f_b}\}
\]
is an invisibility cloaking device for the region $D$ with respect
to the reference space $\{\Omega; \mathbf{A_b, B_b,f_b}\}$. That is,
for any extended object
\[
\{\Omega;\mathbf{A_e,B_e, f_e}\}=\begin{cases}
\{\Omega\backslash\bar{D}; \mathbf{A_c,B_c,f_c}\}\ &\mbox{in
\ $\Omega\backslash\bar{D}$}, \\
\{D; \mathbf{A_a,B_a,f_a}\}\ &\mbox{in $\ D$},
\end{cases}
\]
where $\{D; \mathbf{A_a, B_a}\}$ is an arbitrary regular medium and
$\mathbf{f_a}\in {\widetilde H}^{-1}(D)^m$, we have
\begin{equation}\label{eq:equalcloaking general}
\Lambda_{\mathbf{A_e,B_e,f_e}}^\omega=\Lambda_{\mathbf{A_b,B_b,f_b}}^\omega.
\end{equation}
\end{thm}

\begin{proof}
Let
\begin{equation}
\{\mathbf{B}_2; A_b, B_b, f_b\}:=(G^{-1})_*\{\Omega; \mathbf{A_b,
B_b, f_b}\},
\end{equation}
and
\begin{equation}
\begin{split}
\{\mathbf{B}_2\backslash\bar{\mathbf{B}}_1; A_c, B_c, f_c\}=&(F\circ
G^{-1})_*\{\Omega\backslash G(\{0\}); \mathbf{A_b, B_b, f_b}\}\\
=& F_*\{\mathbf{B}_2\backslash\{0\}; A_b, B_b, f_b\}.
\end{split}
\end{equation}
Then,
\[
\{\Omega\backslash\bar{D}; \mathbf{A_c, B_c,
f_c}\}=G_*\{\mathbf{B}_2\backslash\bar{B}_1; A_c, B_c, f_c\}.
\]
Since $G$ is orientation-preserving and bi-Lipschitz, we know
$\{\mathbf{B}_2;A_b,B_b,f_b\}$ is a regular space and, $\mathbf{A_c,
B_c}$ have same singularities on the cloaking interface $\partial
D^+$ as of $A_c, B_c$ on $\mathbf{S}_1^+$. In order to stick close
to our discussion in Sections 2 and 3, we introduce
\[
\{\mathbf{B}_1; A_a, B_a, f_a\}=(G^{-1})_*\{D; \mathbf{A_a, B_a,
f_a}\},
\]
and define $\{\mathbf{B}_2; A_e, B_e, f_e\}$ correspondingly. Now we
introduce the singulary weighted Sobolev space
$H^1_{[\mathbf{A_e,B_e,f_e}]}(\Omega)^m$ similarly to
$H^1_{[A_e,B_e]}(\mathbf{B}_2)^m$. The solution to the cloaking
problem,
\begin{equation}\label{eq:f01}
\mathcal{P}_{[\mathbf{A_e,B_e}]}\mathbf{u}=\mathbf{f_e}\ \ \mbox{on
\ $\Omega$},\ \ \mathbf{u}|_{\partial\Omega}=\mathbf{h}
\end{equation}
can be defined in a similar manner as in (\ref{eq:distributional
equation}). It is readily seen that $\mathbf{u}\in
H^1_{[\mathbf{A_e, B_e}]}(\Omega)^m$ is a solution to (\ref{eq:f01})
if{f} $u:=G^*\mathbf{u}\in H^1_{[A_e,B_e]}(\mathbf{B}_2)$ is a
solution to
\begin{equation}
\mathcal{P}_{[A_e,B_e]}u=f_e\ \ \mbox{on\ $\mathbf{B}_2$},\ \
u|_{\mathbf{S}_2}=h
\end{equation}
where $h=(G|_{\mathbf{S}_2})^* \mathbf{h}$. By Theorem~\ref{thm:2},
we know $\tilde
v(x)=(F^*(u|_{\mathbf{B}_2\backslash\bar{\mathbf{B}}_1}))^e \in
H^1(\mathbf{B}_2)^m$ is satisfies
\begin{equation}\label{eq:f11}
\mathcal{P}_{[A_b,B_b]}\tilde v=f_b\ \ \mbox{on\ $\mathbf{B}_2$},\ \
\tilde v|_{\mathbf{S}_2}=h,
\end{equation}
whereas $w=u|_{\mathbf{B}_1}\in H^1(\mathbf{B}_1)^m$ is a solution
to
\begin{equation}\label{eq:f12}
\mathcal{P}_{[A_a,B_a]}w=f_a\ \ \mbox{on\ $\mathbf{B}_1$},\ \
w|_{\mathbf{S}_1}={c}_0,
\end{equation}
with ${c}_0\in\mathbb{C}^m$ a constant vector determined by
\begin{equation}\label{eq:f13}
\int_{\mathbf{S}_1}\sum_{\alpha,\beta\leq 1}\nu_\alpha
A_a^{\alpha\beta}\partial_\beta w\ dS=0.
\end{equation}
Therefore, we know $\tilde{\mathbf{v}}:=(G^{-1})^*\tilde
v=(\mathbf{K}^*(\mathbf{u}|_{\Omega\backslash\bar{D}}))^e\in
H^1(\Omega)^m$ satisfies
\begin{equation}
\mathcal{P}_{[\mathbf{A_b, B_b}]}\tilde{\mathbf{v}}=\mathbf{f_b}\ \
\mbox{on\ $\Omega$},\ \
\tilde{\mathbf{v}}|_{\partial\Omega}=\mathbf{h},
\end{equation}
whereas $\mathbf{w}=:(G^{-1})^*w=\mathbf{u}|_{D}\in H^1(D)^m$ is a
solution to
\[
\mathcal{P}_{[\mathbf{A_a},\mathbf{B_a}]}\mathbf{w}=\mathbf{f_a}\ \
\mbox{on\ $D$},\ \ \mathbf{w}|_{D}=\mathbf{c}_0,
\]
with $\mathbf{c}_0\in\mathbb{C}^m$ a constant vector determined by
\[\int_{\partial D}\sum_{\alpha,\beta\leq 1}\nu_\alpha
\mathbf{A_a}^{\alpha\beta}\partial_\beta \mathbf{w}\ dS=0.
\]

Finally, by a similar argument to that for the proof of
Theorem~\ref{thm:1}, one can show (\ref{eq:equalcloaking general}).
The proof is completed.
\end{proof}

\section{Discussion}

In this paper, we investigate the invisibility cloaking for a
general system of PDEs in $\mathbb{R}^2$. The material parameters
possess transformation properties, which allow the construction of
cloaking media by singular push-forward transformations. The
cloaking media are inevitably singular resulting in singular PDEs.
Thus one needs to be careful in defining the meaning of solutions to
the singular cloaking problems. In \cite{GKLU} the notion of finite
energy solutions was introduced which is a meaningful physical
condition because it relates to the energy of the waves. We follow
their approach to treat more singular PDEs in the present paper. By
studying the finite energy solution, we decouple the cloaking
problem with one in the cloaking region having the same DtN operator
as that in the background/reference space, and the other one in the
cloaked region possessing some `hidden' boundary condition on the
interior cloaking interface. On the other hand, we would like to
point out that this is not the unique way of defining solutions to
the singular cloaking problems. An alternative treatment is provided
in \cite{HetmaniukLiuZhou} by the authors to split the finite energy
solutions space directly and hence decouple the cloaking problem
into the cloaked region and cloaking region automatically. In order
to achieve more physical insights into these solutions, one might
relate them to solutions from regularized approximate cloaking.

\end{document}